\documentclass[11pt,reqno,final]{article}

\usepackage[color,notref,notcite]{showkeys}
\definecolor{labelkey}{rgb}{0.6,0,1}
\usepackage{epsfig}
\usepackage{a4wide,amssymb,amsmath}
\usepackage{amsthm}
\usepackage{color}
\usepackage{url}
\usepackage{mathtools}

\usepackage[applemac]{inputenc}				
\usepackage[english]{babel}

\usepackage{amsfonts}
\usepackage{epsfig}
\usepackage{graphicx}

\input epsf

\def\R{\hbox{\bf R}}

\def\1{1}

\def\I{{\cal I}}

\def\O{\Omega}

\newcommand{\ba}{\begin{eqnarray}}
\newcommand{\ea}{\end{eqnarray}}

\newtheorem{theo}{\bf Theorem}[section]
\newtheorem{pro}[theo]{\bf Proposition}
\newtheorem{cor}[theo]{\bf Corollary}
\newtheorem{defi}[theo]{\bf Definition}

\newtheorem{lemma}{\textbf{Lemma}}[section]
\newtheorem{remark}[theo]{\bf Remark}
\providecommand{\keywords}[1]{\textbf{\textit{Keywords: }} #1}

\renewcommand{\R}{{\mathbb R}} 

\newenvironment{Proofc}[1]{\smallskip\par\noindent\textsc{#1}\quad}%
  {\hfill$\Box$\bigskip\par}

\newcommand{\Usc}{{\textrm{USC}}}
\newcommand{\Lsc}{{\textrm{LSC}}}

\newcommand{\tf}{\mathtt{f}} 
\newcommand{\cS}{\mathcal{S}} 
\newcommand{\cH}{{\mathcal H}} 
\newcommand{\cI}{{\mathcal I}}
\newcommand{\cO}{{\mathcal O}}

\newcommand{\vN}{N}	
\newcommand{\vx}{x}      
\newcommand{\vR}{R}	
\newcommand{\vH}{H}     
\newcommand{\vV}{V}	

\newcommand{\vomega}{\omega}  

\title{\bf
A unified approach to the well-posedness  of some  non-Lambertian models in Shape-from-Shading theory
}

\author{Fabio Camilli\footnotemark[1]\ \and Silvia Tozza\footnotemark[2]}
\begin{document}

\maketitle

\renewcommand{\thefootnote}{\fnsymbol{footnote}}

\footnotetext[1]
{
 Dip. di Scienze di Base e Applicate per l'Ingegneria,  ``Sapienza" Universit{\`a}  di Roma,
 via Scarpa 16, 00161 Rome, Italy, ({\tt e-mail:camilli@sbai.uniroma1.it})
}

\footnotetext[2]
{
Dipartimento di Matematica,  ``Sapienza" Universit{\`a}  di Roma,
P.le Aldo Moro, 5 - 00185 Rome, Italy
({\tt e-mail: tozza@mat.uniroma1.it})
}

\renewcommand{\thefootnote}{\arabic{footnote}}

\abstract{In this paper we show that the introduction of an attenuation factor in the 
brightness equations relative to various perspective Shape-from-Shading models allows  to make the corresponding differential problems well-posed. We propose a unified approach based on the theory of viscosity solutions and we show that  the brightness
equations with the attenuation term admit a unique viscosity solution.
We also discuss in detail the possible boundary conditions that we can use for the Hamilton-Jacobi equations associated to these models.
}
\vskip12pt


\section{Introduction}\label{sec:introduction}
Shape-from-Shading  (SfS) is the problem to compute a three dimensional shape of a surface from a  single gray-value image of it. The SfS model was formally introduced by Horn  who first formulated the problem as  a nonlinear first order partial differential equation (PDE) of Hamilton-Jacobi (HJ) type.
In general, this problem is described by the following so-called \emph{image irradiance equation} introduced by Bruss \cite{Bruss81}
\begin{equation}\label{general_irradiance_eq}
I(\vx) = R(\vN(\vx)).
\end{equation}
In this equation, the normalized brightness $I(\vx)$ of the given gray-value image  is put in relation with the reflectance map $R(\vN(\vx))$
 giving the value of the light reflection on the surface as a function of its orientation (i.e., of the normal $\vN(\vx)$).
Depending on how we describe the function $R$, different reflection models are determined.
The most studied model is the Lambertian one, a view independent model that depends only on the incident angle between the normal and the light source direction $\vomega$. Considering this model under an orthographic projection, that is when a single light source is placed at infinity in the direction of the unit vector $\vomega$
and  assuming uniform albedo equal to $1$ (i.e. the light is completely reflected by the surface)
the irradiance equation \eqref{general_irradiance_eq} becomes
\begin{equation}\label{intro_1}
    I(\vx)\sqrt{1+|\nabla u(\vx)|^2}+\vomega \cdot (\nabla u(\vx),-1)=0,
\end{equation}
where $u$ is the surface height (our unknown).
In general, nonlinear PDEs such as \eqref{intro_1}
do not admit smooth solutions and the proper notion of weak solution   is the viscosity  one (see \cite{Barles94,CIL,Bardi_Capuzzo_2008}).
Viscosity solutions  in the framework of SfS were considered for the first time in \cite{RT} and since then many papers have been
devoted to develop a proper existence and uniqueness theory for \eqref{intro_1} and related models.
It is  worthwhile to note that the orthographic  Lambertian model suffers of two main drawbacks:\\
 \emph{(i)}  From a modeling point of view, the Lambertian model does not take into account the viewer direction and supposes to deal with smooth objects, assumptions not suitable to deal with real-world data. Moreover, except for few application fields such as  the astronomical one, the light source and the camera are not far away from the surface (e.g. applications in medicine, security, etc.). \\
 \emph{(ii)} The image irradiance equation \eqref{intro_1}, even if completed with boundary conditions, does not admit in general a unique solution. In fact, due to the well-known concave/convex ambiguity in presence of \emph{singular  points}, i.e. points where
 the light direction is parallel to the normal to the surface, there exists an infinite family  of   viscosity solutions to equation  \eqref{intro_1}
 which are  all in between a minimal and a maximal solution. \par
 Concerning the first issue, in order to increase the complexity of the model, various extensions have been proposed.
In \cite{PF05,OD96},  the orthographic Lambertian  model is replaced by a \emph{perspective Lambertian one} (pinhole camera and light source at the optical center). In \cite{JTBBK13}, a more general setup,   which combines a spherical surface parametrization with the non-Lambertian Oren-Nayar reflectance model, is considered  obtaining a robust approach that allows to deal with an arbitrary position of the light source.
Moreover, models for non-Lambertian surfaces (like e.g. \cite{Phong75,Blinn77,ON94}) have been proposed in \cite{AF06,AF07,VBW08,JTBBK13}.
All these models, when formulated in terms of a differential equation, do not resolve the concave/convex ambiguity (see \cite{TF16} for an analysis on them). \par
In order to solve the non-uniqueness issue in presence of singular points, in \cite{PFC04,PF05} the authors  introduce   in the perspective Lambertian model an attenuation factor which takes into account the distance between the light source and the surface. This new term   makes
the model more reliable and thanks to it the notion of singular points does not make  sense anymore (even if some ambiguities in the model still appear \cite{BCDFV12}). They show that a Hamilton-Jacobi equation, obtained
via an appropriate change of variable  in the image irradiance equation,  admits a unique viscosity solution and, therefore, the uniqueness of the solution for the original equation follows. \par
In this paper,  starting from the idea in  \cite{PF05} related to the Lambertian model, we consider non-Lambertian models (Oren-Nayar \cite{ON94},
Phong \cite{Phong75}, Blinn-Phong \cite{Blinn77}) and we show that the  introduction of  an attenuation factor in the perspective case
 allows to overcome the concave/convex ambiguity for  all the corresponding brightness equations  and to obtain therefore
 the existence of at most one viscosity solution.  The uniqueness  property is obtained by checking the assumptions of the \emph{Maximum Principle for discontinuous viscosity solutions}  (see \cite{Barles94}). In this framework, boundary conditions  play  an important role,
hence we discuss  the various type of  them  which is possible to prescribe for the different models.
Interesting enough, all the  models  (included the  Lambertian one) can be studied in a unified approach
which allows to check the assumptions of the Maximum Principle in a rather straightforward way. Moreover, this approach  can  also provide
 a setup    to study other SfS  perspective models (see \cite{AF07}).\par   
 Many papers have been devoted to the approximation of the various irradiance equations arising in SfS theory (see the surveys \cite{ZTCS99,DFS08}), but only
very few of them discuss the convergence of the corresponding methods  and most of the results are of experimental type.
This fact is essentially due to the lack of a proper characterization of the viscosity  solution of the limit equation. The   Maximum Principle for discontinuous viscosity solutions we prove here is   a key ingredient to show  the convergence of numerical schemes for SfS equations via the by now classical method  proposed  in \cite{BS} (we will discuss this point elsewhere).\par
The paper is organized as follows: in Section \ref{sec:models}, we briefly describe the reflectance models and the derivation of the PDEs associated to them, arriving to a unified formulation with common properties shared by the different models.
In Section \ref{sec:visco_solutions_HJEs}, we focus our attention on viscosity solutions and the role of the various possible boundary conditions for HJ equations. After that, we move to prove the well-posedness of the Perspective SfS problem for the non-Lambertian models in Section \ref{sec:well-posedness}.
Finally, we report final comments in Section \ref{sec:conclusions} and  in Appendix \ref{App:AppendixA}  the  proofs of two lemmas used for  the  unified analysis.


\section{Reflectance Models and PDE formulation}\label{sec:models}
In this section we will briefly recall the reflectance models we will consider and we will describe the derivation of the corresponding HJ  equations, arriving to a unified formulation of them. 
For all the models, we consider as setup the perspective camera projection with one light source located at the optical center of the camera
and an attenuation term of the illumination due to the distance between the light source and the surface.\par
Let $\Omega$ be an open subset of $\mathbb{R}^2$ representing the image domain. Since the CCD sensors have finite size, let us assume that $\Omega$ is bounded.
As in \cite{PFC04,PF05,AF06,VBW08}, we represent the scene by a surface parametrized by   the function $\mathcal{S}: \overline{\Omega} \rightarrow \mathbb{R}^3$ defined by
\begin{equation*}\label{S_Surface}
\mathcal{S}(\vx) = \frac{f u(\vx)}{\sqrt{|\vx|^2 + f^2}} (\vx, -f)
\end{equation*}
where $f>0$ is the focal length of the camera and $u(\vx)$ is the unknown height of the surface.\\
The normal vector   at the point $\vx := (x_1,x_2)$ of $\mathcal{S}$ is given by
\begin{equation}\label{normal}
n(\vx) = \mathcal{S}_{x_1} \times \mathcal{S}_{x_2}= \big( f \, \nabla u(\vx) - \frac{f u(\vx)}{|\vx|^2 + f^2} \, \vx, \nabla u(\vx)\cdot \vx +  \frac{f u(\vx)}{|\vx|^2 + f^2} \, f \big)
\end{equation}
and the unit vector in the direction of the illumination is
\begin{equation*}\label{light_source}
 \vomega(\mathcal{S}(\vx))  = \frac{(-\vx,f)}{\sqrt{|\vx|^2 + f^2}} \, .
\end{equation*}
The unit normal vector $\vN(\vx)$ will be $\vN(\vx) = \frac{n(\vx)}{|n(\vx)|}$, where
\begin{eqnarray}\label{eq:norm_n}
|n(\vx)|
 =&\displaystyle \sqrt{ f^2|\nabla u(\vx)|^2 + (\nabla u(\vx) \cdot \vx)^2 + u(\vx)^2 \frac{f^2}{(f^2 + |\vx|^2)} }. \label{norm_of_n}
\end{eqnarray}

\subsection{Lambertian model (L--model)}\label{subsec:L_model}
The Lambertian model is the most common   reflectance model (\cite{PF05,OD96,HB86,HB89}).
The brightness equation for this model is described by
\begin{equation}\label{L_reflectance_model}
I(\vx)= \gamma_D(\vx) \frac{\cos(\theta _i)}{r^2}
\end{equation}
where $I(\vx)$ denotes the image intensity,
$\theta_{i}$ represents the angle between the unit normal to the surface $\vN(\vx)$ and the light source direction $\vomega(\mathcal{S}(\vx))$, $\gamma_D(\vx)$ is the diffuse albedo which describes the physical properties of the surface reflection, and $1/r^2$ is the attenuation term of the illumination due to the distance between the light source and the surface \cite{AF06}.
Hence, for a Lambertian surface, the measured light in the image depends only on the scalar product between $\vN$ and $\vomega$, independently on the position of the viewer.
In the sequel, we suppose uniform albedo and we put $\gamma_D(\vx)=1$, that is all the points of the surface reflect completely the light that hits them (we will not consider absorption phenomena).
So in this case the surface of the object has uniform properties of light reflection (see Fig. \ref{fig:lamb_reflection}).
\begin{figure}[h!]
\centering
\includegraphics[width=6.6cm]{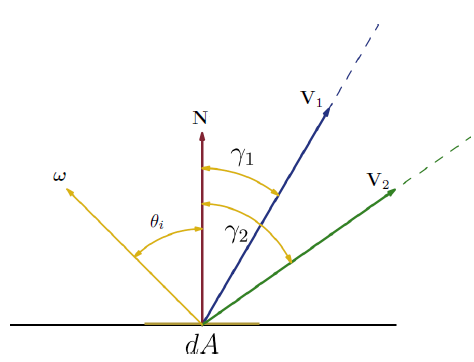}
\caption[A Lambertian surface diffuses the incident radiation independently from the angle between the surface normal and the direction of the observer.]{
A Lambertian surface diffuses the incident radiation independently from the angle between the surface normal $\vN$ and the direction of the observer.
Two different observers $\vV_1$ and $\vV_2$ do not detect any difference in radiance. The radiance is a function only of the angle $\theta_i$ between $\vomega$ and $\vN$.}
\label{fig:lamb_reflection}
\end{figure}
Under a perspective camera projection with a single light source located at the optical center of the camera and considering the attenuation term $1/r^2$, with $r = f u(\vx)$, we can associate to \eqref{L_reflectance_model} the following HJ equation
\begin{equation}\label{PDE_perspective_L}
I(\vx) = \frac{1}{f^2 u(\vx)^2} \displaystyle \frac{u(\vx) Q(\vx)}{\sqrt{f^2 |\nabla u(\vx)|^2 + (\nabla u(\vx) \cdot \vx)^2 + u(\vx)^2 Q(\vx)^2}}
\end{equation}
where
\begin{equation}\label{def_Q}
Q(\vx) := \displaystyle \frac{f}{\sqrt{f^2 + |\vx|^2}}.
\end{equation}
By defining
\begin{equation}\label{def_W}
W(\vx, p) := \displaystyle \frac{f^2 |p|^2 + (p \cdot \vx)^2}{Q^2},
\end{equation}
we can rewrite Eq. \eqref{PDE_perspective_L} as
\begin{equation*}
I(\vx) = \frac{1}{f^2 u(\vx)^2} \displaystyle \frac{u(\vx)}{\sqrt{W(\vx, \nabla u(\vx)) + u(\vx)^2}}.
\end{equation*}
Assuming that the surface $\cS$ is visible, i.e. in front of the optical center, then $u$ is a strictly positive function. Hence, by the change of variable $v(\vx) = \ln u(\vx)$, we  get the following HJ equation in the new variable $v$ (as in \cite{PFC04,PF05})
\begin{equation} \label{L_pde_v_PSFS}
 -e^{-2v(\vx)} + I(\vx) f^2 \sqrt{W(\vx, \nabla v(\vx)) + 1} = 0,
\end{equation}
to which we  associate the  Hamiltonian
\begin{equation} \label{eq:L_Hamiltonian}
H^L(\vx, r, p) = -e^{-2r} + I(\vx) f^2 \sqrt{W(\vx, p) + 1} \, .
\end{equation}
\subsection{Oren-Nayar model (ON--model)}\label{subsec:ON_model}
The Oren-Nayar reflectance model \cite{ON94,ON95} represents a rough surface as an aggregation of V-shaped cavities, each with Lambertian reflectance properties.
Assuming that there is a linear relation between the irradiance of the image and the image intensity, the   brightness equation  for the ON--model is given by
\begin{equation}
\label{ON_reflectance_model}
I(\vx)= \cos(\theta _i) (A + B \sin(\alpha) \tan (\beta )M(\varphi_i,\varphi_r)).
\end{equation}
In this model, 
$\theta_{i}$ represents the angle between the unit normal to the surface $\vN$ and the light source direction $\vomega$,
$\theta_{r}$ stands for the angle between $\vN$ and the observer direction $\vV$,
$\varphi_{i}$ is the angle between the projection of the light source direction $\vomega$ and the $x_1$
axis onto the $(x_1, x_2)$-plane,
$\varphi_{\mathrm r}$ denotes the angle between the projection of the observer direction $\vV$ and the $x_1$ axis
onto the $(x_1, x_2)$-plane, the two variables $\alpha$ and $\beta$ are given by
\begin{equation*}\label{def:ab}
\alpha = \max \left\{ \theta_{i}, \theta_{r} \right\} \hbox{ and } \beta = \min \left\{ \theta_{i}, \theta_{r} \right\}.
\end{equation*}
The nonnegative  constants $A$ and $B$  depend  on the statistics of the cavities via the roughness parameter $\sigma$. We  set $\sigma \in [0,\pi/2)$, interpreting $\sigma$ as the slope of the cavities, and
\begin{eqnarray*} \label{eq:A_B}
&& A = 1 - 0.5 \, \sigma^2 (\sigma^2 + 0.33)^{-1} \\
&& B = 0.45\sigma^2(\sigma^2 + 0.09)^{-1} \\
&& M(\varphi_i,\varphi_r) = \max\{0,\cos(\varphi_r -  \varphi_i)\}.
\end{eqnarray*}
Let us consider a perspective camera projection, with only one light source located at the optical center of the camera. As a consequence of this setup, we define $\theta := \theta_i = \theta_r = \alpha = \beta$. Adding the attenuation term $1/r^2$ of the illumination due to the distance between the light source and the surface as proposed in \cite{AF06}, the brightness equation becomes
\begin{equation}\label{ON_perspective}
    I(\vx) = \frac{A \cos(\theta) +  B \sin^2(\theta)}{r^2}.
\end{equation}
Since $\cos(\theta) = \vN(\vx) \cdot \vomega(\mathcal{S}(\vx))$, hence $\sin^2(\theta) = 1 - (\vN(\vx) \cdot  \vomega(\mathcal{S}(\vx)))^2$, and $r = f u(\vx)$,  the brightness equation \eqref{ON_perspective} becomes
\begin{eqnarray}\label{PDE_perspective_ON}
  I(\vx) f^2 u(\vx)^2 = &&  \displaystyle \frac{A u(\vx) Q(\vx)}{\sqrt{f^2 |\nabla u|^2 + (\nabla u \cdot \vx)^2 + u(\vx)^2 Q(\vx)^2}}  \nonumber \\
   && +  B \Big(1 -\displaystyle \frac{u(\vx)^2 Q(\vx)^2}{f^2 |\nabla u|^2 + (\nabla u \cdot \vx)^2 + u(\vx)^2 Q(\vx)^2}\Big)
\end{eqnarray}
with $Q(\vx)$ defined as in \eqref{def_Q}.  
By defining $W$ as in \eqref{def_W}, 
we can rewrite Eq. \eqref{PDE_perspective_ON} as
\begin{eqnarray*}\label{PDE_persp_con_W}
 I(\vx) f^2 u(\vx)^2 - \displaystyle \frac{A u(\vx)}{\sqrt{W(\vx,\nabla u(\vx)) + u(\vx)^2}}  - B (1 -\displaystyle \frac{u(\vx)^2 }{W(\vx,\nabla u(\vx)) + u(\vx)^2}) = 0
\end{eqnarray*}
or, equivalently,
\begin{eqnarray*}
 I(\vx) f^2 u(\vx)^2 (W (\vx,\nabla u(\vx)) + u(\vx)^2) - A u(\vx) (\sqrt{W(\vx,\nabla u(\vx)) + u(\vx)^2}) - B W(\vx,\nabla u(\vx)) = 0. \nonumber
\end{eqnarray*}
Since  $u(\vx) > 0, \, \forall \vx \in \overline \Omega$,  
by the change of variable $v(\vx) = \ln u(\vx)$
we arrive to the following PDE in $v$
\begin{equation*} %
 I f^2 e^{2v(\vx)} (W(\vx, \nabla v(\vx)) + 1) - A \sqrt{W(\vx, \nabla v(\vx)) + 1} - B W(\vx, \nabla v(\vx))  = 0,
\end{equation*}
which  can be written also as follow (Cf. Eq. (10) in \cite{AF06}):
\begin{equation} \label{ON_pde_v_PSFS}
 -e^{-2v(\vx)} + I(\vx) f^2 \frac{(W(\vx, \nabla v(\vx)) + 1)}{ A \sqrt{W(\vx, \nabla v(\vx)) + 1} + B W(\vx, \nabla v(\vx))} = 0,
\end{equation}
to which we associate the following Hamiltonian
\begin{equation}
  \label{eq:ON_Hamiltonian}
  H^{ON}(\vx, r, p)
  :=
  -  e^{-2r} + I(\vx) \, \tf^2
  \left[
      \dfrac{(W(x,p) +1)}{A \, \sqrt{W(x,p) +1} + B W(x,p) } \,
  \right].
\end{equation}
Note that the quantity $W$ is the same one defined in the L--model, see \eqref{def_W},  and this will be useful for the unified approach to the different models.


\subsection{Phong model (PH--model)}\label{subsec:PH_model}
The PH--model is an empirical model that was developed by Phong \cite{Phong75}.
This model introduces a specular component to the brightness function $I(\vx)$ which can be described in general as
\[  k_A I_A(\vx) + k_D I_D(\vx) + k_S I_S(\vx)\]
 where $I_A(\vx)$, $I_D(\vx)$ and $I_S(\vx)$ are the ambient, diffuse and specular light components, respectively, and $k_A$, $k_D$ and $k_S$ indicate the percentages of ambient, diffuse and specular components, respectively, such that their sum is equal to $1$.
 The specular light component $I_S(\vx)$ is given as a power of the cosine of the angle $\theta_s$ between the unit vectors $\vV$ and $\vR(\vx)$
with $R$ representing the reflection of the light $\vomega$ on the surface. Then, the brightness equation for the PH--model is
\begin{equation}
\label{eq:PH_brightness}
I(\vx) = k_A I_A(\vx) + k_D \gamma_D(\vx) (\cos \theta_i) + k_S \gamma_S(\vx) (\cos \theta_s)^{\alpha},
\end{equation}
where $\alpha$ expresses the specular reflection characteristics of a material, $\gamma_D(\vx)$ and $\gamma_S(\vx)$ represent the diffuse and specular albedo, respectively (we will set them to 1 in the following).\par
Using the same setup adopted for the previous models, that is a perspective camera projection with only one light source located at the optical center of the camera, we   obtain that the view direction and the light source direction are the same, hence   $\theta_s = 2 \theta_i$ (see \cite{VBW08}). By adding the light attenuation term as done before, the brightness equation \eqref{eq:PH_brightness} becomes
\begin{equation}
\label{eq:PH_reflectance_model}
I(\vx) = k_A I_A(\vx) + \frac{1}{r^2} k_D (\vN(\vx) \cdot \vomega(\mathcal{S}(\vx))) + k_S (2 (\vN(\vx) \cdot \vomega(\mathcal{S}(\vx)))^2 - 1 )^{\alpha},
\end{equation}
since $\cos \theta_i = \vN(\vx) \cdot \vomega(\mathcal{S}(\vx))$ and $\cos \theta_s = \cos 2\theta_i = 2(\cos \theta_i)^2 - 1 = 2 (\vN(\vx) \cdot \vomega(\mathcal{S}(\vx)))^2 -1$.\par
Following \cite{VBW08}, we derive a HJ equation associated to \eqref{eq:PH_reflectance_model}. If we use $|n(\vx)|$ as defined in \eqref{eq:norm_n} and $Q(\vx)$ as in \eqref{def_Q}, then \eqref{eq:PH_reflectance_model} is equivalent to
\begin{equation}\label{eq:PDE_perspective_PH}
I(\vx) = k_A I_A(\vx) + \frac{1}{f^2 u(\vx)^2} \Big( k_D \frac{u(\vx) Q(\vx)}{|n(\vx)|} + k_S (2 \frac{u^2 Q(\vx)^2}{|n(\vx)|^2} - 1 )^{\alpha} \Big).
\end{equation}
Since $u(\vx)$ is a strictly positive function, by   the  change of variable  $v(\vx) = \ln u(\vx)$
we arrive to the following HJ equation in $v$
\begin{equation}\label{PH_PDE_v_Vogel}
\begin{array}{cc}
\displaystyle \frac{(I(\vx) - k_A I_A(\vx)) f^2}{Q}   W^{VBW}(\vx, \nabla v(\vx)) - \displaystyle k_D e^{-2v(\vx)} \\
 -\displaystyle  \frac{k_S W^{VBW}(\vx,\nabla v(\vx))}{Q(\vx)} e^{-2v(\vx)} \Big( \frac{2 Q(\vx)^2}{W^{VBW}(\vx,\nabla v(\vx))^2} - 1 \Big)^{\alpha} = 0,
\end{array}
\end{equation}
where
\[W^{VBW}(\vx, p) := \sqrt{f^2 |	p|^2 + (p \cdot \vx)^2 + Q(\vx)^2},\]
as defined in \cite{VBW08}.
To our purpose, it is useful to write \eqref{PH_PDE_v_Vogel} with the same structure and functions that appear in the L-model  and in the ON-model. Hence, we rewrite \eqref{PH_PDE_v_Vogel} as follow
\begin{equation}\label{PH_PDE_v_PSFS}
-e^{-2v(\vx)} + (I(\vx) - k_A I_A(\vx)) f^2 \displaystyle \frac{\sqrt{W(\vx, \nabla v(\vx)) + 1}}{k_D + k_S \sqrt{W(\vx, \nabla v(\vx)) + 1} \, \Big(\frac{1-W(\vx, \nabla v(\vx))}{1+W(\vx, \nabla v(\vx))}\Big)^{\alpha}} = 0,
\end{equation}
where $W(\vx, \nabla v(\vx))$ is defined as in \eqref{def_W}. We associate to \eqref{PH_PDE_v_PSFS} the  following Hamiltonian
\begin{equation}  \label{eq:PH_Hamiltonian_old}
 H^{PH} (\vx,r,p) :=
 -e^{-2r} + (I(\vx) - k_A I_A(\vx)) f^2 \displaystyle \frac{\sqrt{W(\vx, p) + 1}}{k_D + k_S \sqrt{W(\vx, p) + 1} \, \Big(\frac{1-W(\vx, p)}{1+W(\vx, p)}\Big)^{\alpha}}.
\end{equation}
 From now on we assume that $\alpha$ is an integer  and $\alpha\ge 1$ (in \cite{Phong75}, $\alpha \in [1,10]$).
Note that  $H^{PH} (\vx,r,p)$ is not well defined  if $\alpha$ is irrational.
On the other side, as stated also in Vogel et al. \cite{VBW08}, since the term $\Big(\frac{1-W(\vx, p)}{1+W(\vx, p)}\Big)$ represents the cosine of the specular term, in the Phong model this cosine is replaced by zero if $\cos \theta_s = \Big(\frac{1-W(\vx, p)}{1+W(\vx, p)}\Big) < 0$ and we can rewrite \eqref{eq:PH_Hamiltonian_old} as follow:
\begin{eqnarray}  \label{eq:PH_Hamiltonian}
&& H^{PH}(\vx,r,p) := \nonumber \\
&& \left\{
\begin{array}{ll}
-e^{-2r} + (I(\vx) - k_A I_A(\vx)) f^2 \displaystyle \frac{\sqrt{W(\vx, p) + 1}}{k_D} & \textrm{if $W(\vx,p) \ge 1$}, \\
 -e^{-2r} + (I(\vx) - k_A I_A(\vx)) f^2 \displaystyle \frac{\sqrt{W(\vx, p) + 1}}{k_D + k_S \sqrt{W(\vx, p) + 1} \, \Big(\frac{1-W(\vx, p)}{1+W(\vx, p)}\Big)^{\alpha}} &\textrm{if $0 \le W(\vx, p) < 1$}.
\end{array}
\right .
\end{eqnarray}
\begin{remark}\label{PH_reduces_to_L}
Note that if $k_S=0$,  i.e. if there is no contribution of the specular part, then
 \[H^{PH}(x,r,p)=-e^{-2r} +\left(\frac{I(x)-k_A I_A(x)}{k_D}\right)f^2\sqrt{W(\vx, p) + 1}.\]
 Hence, in this case the PH--model reduces to   the L--model, up to a constant.
\end{remark}
\subsection{Blinn-Phong model (BP--model)}\label{subsec:BP_model}
A modification of the PH--model has been proposed by Blinn \cite{Blinn77} introducing an intermediate vector $\vH$, which bisects the angle between the unit vectors $\vomega$ and $\vV$.
If the surface is a perfect mirror, the light reaches the eye only if the surface normal $\vN$ is pointed halfway between the light  source direction $\vomega$ and the eye direction $\vV$. We will denote the direction of maximum highlight  by $\vH = \frac{\vomega + \vV}{|\vomega + \vV|}$.
For less than perfect mirrors, the specular component falls off slowly as the normal direction moves away from the specular direction. The cosine of the angle between $\vN$ and $\vH$ is used as a measure of the distance of a particular surface  from the maximum specular direction. The degree of sharpness  of the highlights is adjusted by taking this
cosine to some power (in \cite{Blinn77} it is stated that this power is typically $50$ or $60$).\\
For this model, the brightness equation is
\begin{equation}
\label{eq:BP_brightness}
I(\vx) = k_A I_A(\vx) + k_D \gamma_D(\vx) (\cos \theta_i) + k_S \gamma_S(\vx) (\cos \delta)^{c},
\end{equation}
where $c$ expresses a measure of shininess of the surface (we suppose $c\ge 1$) and $\delta$ is the angle between the unit normal $\vN$ and $\vH$.
As  for the PH--model, $\gamma_D(\vx)$ and $\gamma_S(\vx)$ represent the diffuse and specular albedo, respectively, and in what follows we will set them to 1.
Using the same setup adopted for the previous models, we   obtain that the view direction and the light source direction are the same.
By adding the light attenuation factor as done before, the brightness equation \eqref{eq:BP_brightness} becomes
\begin{equation*}
\label{eq:BP_reflectance_model}
I(\vx) = k_A I_A(\vx) + \frac{1}{r^2} k_D (\vN(\vx) \cdot \vomega(\mathcal{S}(\vx))) + k_S (\vN(\vx) \cdot \vomega(\mathcal{S}(\vx)))^{c},
\end{equation*}
since $\cos \theta_i = \vN(\vx) \cdot \vomega(\mathcal{S}(\vx))$ and $\cos \delta$ can be expressed as $\vN(\vx) \cdot \vomega(\mathcal{S}(\vx))$ under this setup.
%
By using \eqref{normal} and \eqref{eq:norm_n} for the definition of the unit normal $\vN(\vx)$,
we   arrive to the following  HJ equation
\begin{equation*}\label{eq:PDE_perspective_BP}
I(\vx) = k_A I_A(\vx) + \frac{1}{f^2 u(\vx)^2} \Big( k_D \frac{u(\vx) Q(\vx)}{|n(\vx)|} + k_S \Big(\frac{u(\vx) Q(\vx)}{|n(\vx)|}\Big)^{c} \Big),
\end{equation*}
with $|n(\vx)|$ defined as  in \eqref{eq:norm_n} and $Q(\vx)$ as in \eqref{def_Q}.
By using the  change of variable  $v(\vx) = \ln u(\vx)$, as done for the previous models, we obtain the following HJ equation in $v$
\begin{equation}\label{BP_PDE_v_old}
\displaystyle (I(\vx) - k_A I_A(\vx)) f^2 e^{2v} = k_D \frac{1}{\sqrt{W(\vx, \nabla v(\vx)) +1}} + k_S \Big( \frac{1}{\sqrt{W(\vx, \nabla v(\vx)) +1}} \Big)^c
\end{equation}
with the function $W(\vx, \nabla v(\vx))$ defined as in \eqref{def_W}. For our purpose, we will rewrite this equation as
\begin{equation}\label{BP_PDE_v_PSFS}
\displaystyle -e^{-2v(\vx)} +  (I(\vx) - k_A I_A(\vx)) f^2 \frac{(W(\vx, \nabla v(\vx)) +1)^{c/2}}{k_D (W(\vx, \nabla v(\vx)) +1)^{\frac{c-1}{2}} + k_S} = 0,
\end{equation}
to which we associate the   Hamiltonian
\begin{equation}  \label{eq:BP_Hamiltonian}
H^{BP}(\vx,r,p) := \displaystyle -e^{-2r} +  (I(\vx) - k_A I_A(\vx)) f^2 \frac{(W(\vx, p) +1)^{c/2}}{k_D (W(\vx,p) +1)^{\frac{c-1}{2}} + k_S}.
\end{equation}
\begin{remark}
For the Blinn-Phong model holds  an  observation similar to the one reported for the PH--model in  Remark \ref{PH_reduces_to_L} , i.e.
for  $k_S=0$  it reduces to the $L$-model, up to a constant.
\end{remark}
\subsection{Unified formulation and analysis} \label{sec:unified_formulation}
As we have seen before, 
the four different reflectance models considered in a common setup lead to
the study of  a  HJ equation of the form
\begin{equation}\label{HJ}
H(\vx,v(\vx), \nabla v(\vx)) = 0,  \quad \textrm{$\forall \vx\in\Omega$,}
\end{equation}
where    the Hamiltonian $H$ is given respectively by \eqref{eq:L_Hamiltonian}, \eqref{eq:ON_Hamiltonian}, \eqref{eq:PH_Hamiltonian}, \eqref{eq:BP_Hamiltonian}
and $v(\vx) = \ln u(\vx)$,  with $u(\vx)$ the  unknown height of the surface.  We   highlight some  properties of the general Hamiltonian
$H$ which are common for all the four cases. First of all, note that \eqref{HJ} can be rewritten as
\begin{equation}\label{HJ1}
   - e^{-2v} +\cH^M(x, \nabla v(\vx)) = 0
\end{equation}
where $M$ is the acronym of the models ($M= L, ON, PH, BP$).
For $W(x,p)$  defined as in \eqref{def_W},  $\cH^M(x,p)$ corresponds for each model to the following:
\begin{itemize}
  \item[] \emph{Lambertian model}
  \begin{align}
  &\cH^{L}(x,p) := I(x)f^2F_L(W(x,p)), \label{cH_L}\\
  & F_{L}(r) = \sqrt{r+1}.		\label{F_L}
  \end{align}
  \item[]  \emph{Oren-Nayar model}
  \begin{align}
& \cH^{ON}(x,p) = I(x)f^2F_{ON}(W(x,p)),	\label{cH_ON}\\
 &F_{ON}(r) = \frac{r+1}{A\sqrt{r+1}+Br}.		\label{F_ON}
  \end{align}
  \item[]  \emph{Phong model}
  \begin{align}
 & \cH^{PH}(x,p) = (I(x) - k_A I_A(x)) f^2 F_{PH}(W(x,p)), 	\label{cH_PH} \\
 & F_{PH}(r) =
         \left\{
    \begin{array}{ll}
    \displaystyle \frac{\sqrt{r+1}}{k_D} &\textrm{if $r \ge 1$} \\
    \displaystyle  \frac{(r+1)^{\alpha + 1/2}}{k_D (r+1)^{\alpha} + k_S (r+1)^{1/2} (1-r)^{\alpha}}  &\textrm{if $0 \le r < 1$}.
    \end{array}
    \right .  \label{F_PH}		
  \end{align}
  \item[]  \emph{Blinn-Phong model}
  \begin{align}
 & \cH^{BP}(x,p) = (I(x) - k_A I_A(x)) f^2 F_{BP}(W(x,p)),	\label{cH_BP}\\
  &F_{BP}(r) =  \frac{(r +1)^{c/2}}{k_D (r +1)^{\frac{c-1}{2}} + k_S}. \label{F_BP}
  \end{align}
\end{itemize}
We give some bounds on $W$ and the various functions $F_M$ above defined which will be useful for our desired unified analysis.
The proofs of these results are reported at the end of the paper in Appendix \ref{App:AppendixA}. 
\begin{lemma}\label{lemma1}
The following bounds hold:
\begin{align}
    &f^2 |p|^2\le W(x,p)\le C|p|^2		\label{hp:W_bounded}\\
    &|D_x W(x,p)|\le C  |p|^2 		\label{hp:Dx_W_bounded}\\
   &|D_p W(x,p)|\le C |p| 			\label{hp:Dp_W_bounded}
\end{align}
with $C$ independent of $(x,p)\in\overline \Omega\times \R^2$.
\end{lemma}
\begin{lemma} \label{lemma2} \hfill\\
\begin{itemize}
 \item[(i)]For $F= F_{L},  \,F_{PH},\,F_{BP}$,    the function $F:[0,+\infty)\to \R$
is smooth, positive,  strictly increasing and satisfies
  \begin{align}
   \lim_{r\to+\infty} F(r) = +\infty,  \label{lim_F}\\
   0<F'(r)\le \frac{C}{\sqrt{r+1}}.  \label{Bound_der_F} 
  \end{align}
 \item[(ii)] The function $ F_{ON}:[0,+\infty)\to \R$ is smooth, positive, strictly increasing  if $A/2>B$  and satisfies
  \begin{align}
   \lim_{r\to+\infty} F_{ON}(r) = \frac{1}{B},	\label{lim_F_ON}\\
   0<F'_{ON}(r)\le \frac{C}{\sqrt{r+1}}. \label{Bound_der_F_ON}
  \end{align}
\end{itemize}
\end{lemma}

\section{Viscosity solutions of Hamilton-Jacobi equations} \label{sec:visco_solutions_HJEs}
In view   of the study of the general Hamiltonian \eqref{HJ} for the various reflectance models introduced in Section \ref{sec:models},
we recall some basic properties of the theory of viscosity solutions emphasizing the role of the boundary conditions
(we refer to \cite{Barles94} for a nice introduction to this theory).\\
In general, a Hamilton-Jacobi equation such as \eqref{HJ} does not admit a classical solution and the
right notion of weak solution  is  the one of viscosity solution. Viscosity solutions are typically Lipschitz continuous and they  can be
obtained as the limit of  regular solutions to second order problems via the so-called vanishing viscosity method. Moreover,
if a classical solution to the HJ equation exists, it coincides with the viscosity solution.
We introduce the definition of viscosity subsolution, supersolution and solution. It is useful to have the following notations
for a generic set $\cO\subset \R^2$
\begin{align*}
    \Usc(\cO)=\{\text{upper semicontinuous functions} \,u:\cO\to\R\}\\
     \Lsc(\cO)=\{\text{lower semicontinuous functions} \,u:\cO\to\R\}
\end{align*}
\begin{defi}\label{visco_sol}
A viscosity subsolution of \eqref{HJ} on $\O$ is a function $u\in\Usc(\O)$ such that
for any $\phi\in C^1(\O)$, if $x_0\in\O$ is a local maximum point  of $u-\phi$, one has
\[H(x_0, u(x_0),\nabla\phi(x_0))\le 0.\]
A viscosity supersolution of \eqref{HJ} on $\O$ is a function $u\in\Lsc(\O)$ such that
for any $\phi\in C^1(\O)$, if $x_0\in\O$ is a local minimum point  of $u-\phi$, one has
\[H(x_0, u(x_0),\nabla\phi(x_0))\ge 0.\]
Finally, $u$ is a viscosity solution of \eqref{HJ} if it is both a viscosity subsolution and a viscosity supersolution.
\end{defi}
\noindent Note that a viscosity solution of \eqref{HJ} is a continuous function in $\O$.
\subsection{Strong boundary conditions}\label{subsec:strong_BC}
In this framework we assume to know the value of the function $u$ on the boundary of the image domain $\O$, hence we consider the
Dirichlet problem:
\begin{equation}\label{Dirichlet}
\left\{
\begin{array}{ll}
H(\vx,u(\vx), \nabla u(\vx)) = 0,  & \textrm{$\forall \vx\in\Omega$,}\\
u(\vx) =g(\vx), & \textrm{$\forall \vx\in\partial\Omega$,}
\end{array}
\right.
\end{equation}
where $g$ is a real continuous function defined on $\partial \Omega$ and the  boundary condition is considered in pointwise  sense (\emph{boundary condition in strong sense}).
The following theorem ensures the uniqueness of the continuous viscosity solution to \eqref{Dirichlet}   (see   \cite[Theorem 3.3]{CIL} for more details).
\begin{theo}\label{theorem_weak_uniqueness}
Let $\Omega$ be  a bounded subset of $\R^2$ and $H: \Omega\times \R \times \R^2 \rightarrow \R$ be continuous and satisfies
\begin{equation}\label{H1}
    \begin{split}
     \text{for any $R>0$, $\exists \gamma_R>0$ s.t. $H(x,u,p)-H(x,v,p)\ge \gamma_R(u-v)$}\\
     \text{for all $x\in\Omega$, $-R\le v\le u\le R$ and $p\in\R^2$,}
    \end{split}	\tag{H1}
\end{equation}
\begin{equation}\label{H2}
    \begin{split}
   \text{$|H(x,u,p)-H(y,u,p)| \le m_R(|x-y|(1+|p|))$ for all $x,y\in \Omega$, }\\
   \text{$-R\le u\le R$, $p\in\R^2$  where $m_R(t)\to 0$ when $t\to 0$.}
    \end{split}	\tag{H2}
\end{equation}
Let $u\in \Usc(\overline\O)$ (respectively  $v\in \Lsc(\overline \O)$) be a subsolution (respectively, supersolution) of  \eqref{HJ}
in $\O$ and $u\le v$ on $\partial\O$. Then $u\le v$ in $\overline\O$.
\end{theo}
The previous result  implies that there exists at most one viscosity solution $u\in C(\overline\O)$ to the Dirichlet problem \eqref{Dirichlet} and we will prove  that all the Hamiltonians  corresponding to the reflectance models introduced in Section \ref{sec:models} satisfy the  assumptions
\eqref{H1}-\eqref{H2}.
Note that the crucial point for the uniqueness is the assumption \eqref{H1}, which is in general not satisfied by the SfS models without attenuation factor and/or under a different setup, e.g. considering an orthographic projection instead of a perspective one.
However, Theorem \ref{theorem_weak_uniqueness} is not completely satisfying for the following reasons:
\begin{itemize}
\item[(i)] A Dirichlet boundary condition requires the exact value of the solution at the boundary of the image domain, an additional information which is not always available in the SfS framework, especially in real contexts.
\item[(ii)] It is well known that  boundary conditions for first order PDEs can be typically prescribed only on some part of the boundary
(active boundary).
Hence, imposing a Dirichlet condition on all the boundary as in  \eqref{Dirichlet} could lead  to an ill-posed problems, i.e.  problems for which there exists no viscosity solution
 (see \cite[cap. IV]{Barles94})  and in this case the corresponding numerical schemes return inconsistent or inexact  results.
\item[(iii)] Many algorithms have been suggested for the SfS problem and a large part of them are based on the approximation of the   partial differential equations arising in the various models (see the surveys \cite{ZTCS99,DFS08}). Convergence and stability analysis    are important points
    for the reliability of a numerical approximation, but only a very few papers analyze the  theoretical properties of various proposed algorithms and most of the results are of empirical nature.\\	
It is worthwhile recall that there is a huge literature about the approximation of viscosity solutions (see e.g. the book by Falcone and Ferretti \cite{FF}) and a well established technique to prove the convergence of a numerical scheme is the Barles-Souganidis' approach \cite{BS}: besides some natural properties of the scheme (stability, consistency, monotonicity), a key ingredient for  this technique is the \emph{Maximum Principle for discontinuous viscosity solution},    which in particular requires the formulation in a weak (viscosity) sense of the  boundary conditions for the PDE to be approximated.
Since the Dirichlet boundary condition in \eqref{Dirichlet} is formulated in strong sense, i.e. pointwise, Theorem \ref{theorem_weak_uniqueness} is not sufficient for applying the Barles-Souganidis' approach and we need a different comparison result.
\end{itemize}
Motivated by the previous remarks, we are going to introduce the notion of boundary conditions in weak (viscosity) sense.
\subsection{Weak boundary conditions}\label{subsec:weak_BC}
In this section we consider a more general boundary value problem of the form
\begin{equation}\label{general_boundary}
\left\{
\begin{array}{ll}
H(x,u(x), \nabla u(x)) = 0,  & \textrm{$\forall x\in\Omega$,}\\
B(x,u(x), \nabla u(x))=0, & \textrm{$\forall x\in\partial\Omega$.}
\end{array}
\right.
\end{equation}
The operator $B:\partial\O\times\R\times\R^2\to\R$ represents the boundary condition. For example, the Dirichlet boundary condition in \eqref{Dirichlet} arises from the choice $B(x,r,p)=r-g(x)$; the Neumann boundary condition arises if $B(x,r,p)=N(x)\cdot p-g(x)$ where $N(x)$ is the outward unit normal to $x\in\partial\O$;
the so-called state constraints condition corresponds to $B(x,r,p)=r-g(x)$ with $g(x)\equiv-\infty$. We introduce the notion of (discontinuous) viscosity solution for \eqref{general_boundary}. The main difference with Definition \ref{visco_sol} is that now the boundary condition is incorporated in the definition (\emph{boundary condition in weak sense}).
\begin{defi}\label{discvisco_sol}
A viscosity subsolution of \eqref{general_boundary}  is a function $u\in\Usc(\overline\O)$ such that
for any $\phi\in C^1(\overline\O)$, if $x_0$ is a local maximum point  of $u-\phi$, one has
\begin{align*}
 &H(x_0, u(x_0),\nabla\phi(x_0))\le 0, &\text{if $x\in  \O$},\\
  &\min[H(x_0, u(x_0),\nabla\phi(x_0)),\,B(x_0u(x_0), \nabla\phi(x_0))]\le 0, & \text{if $x\in \partial \O$.}
\end{align*}
A viscosity supersolution of  \eqref{general_boundary}   is a function $u\in\Lsc(\overline\O)$ such that
for any $\phi\in C^1(\overline\O)$, if $x_0$ is a local minimum point  of $u-\phi$, one has
\begin{align*}
 &H(x_0, u(x_0),\nabla\phi(x_0))\ge 0, &\text{if $x\in  \O$},\\
  &\max[H(x_0, u(x_0),\nabla\phi(x_0)),\,B(x_0u(x_0), \nabla\phi(x_0))]\ge 0,  &\text{if $x\in \partial \O$.}
\end{align*}
Finally, $u$ is a viscosity solution of \eqref{general_boundary} if it is both a viscosity subsolution and a viscosity supersolution.
\end{defi}
\noindent Note that a viscosity solution of \eqref{general_boundary} is a continuous function in $\overline \O$.
We can now state  the \emph{Maximum Principle for discontinuous viscosity solutions}. We first consider the case of a
Dirichlet boundary condition  (see \cite[Theorem 4.4]{Barles94}).
\begin{theo}\label{theorem_strong_uniqueness}
Let $\Omega$ be a  bounded subset of $\R^2$  with $\partial\O$ of class $W^{2,\infty}$, $g:\partial\O\to \R$ a continuous
 function and $H: \Omega\times \R \times \R^2 \rightarrow \R$
satisfies  (H1)-(H2). Moreover, there exists a neighborhood $\Gamma$ of $\partial \Omega$ in $\R^2$ such that
\begin{equation}\label{H3}
    \begin{split}
   \text{for all $0<R<+\infty$, there exists $m_R(t)\to 0$ when $t\to 0$ such that} \\
   \text{$|H(x,u,p)-H(x,u,q)| \le m_R(|p-q|))$, for all $x \in\Gamma$, $-R\le u\le R$, $\, p,q\in\R^2$,             }
    \end{split}	\tag{H3}
\end{equation}
\begin{equation}\label{H4}
    \begin{split}
    \text{for all $0<R< +\infty$, there exists $C_R>0$ such that} \\
   \text{ $H(x,u,p+\lambda n(x))\le 0 \Rightarrow \lambda\le C_R(1+|p|)$ for all $(x,u,p)\in\Gamma\times[-R,R]\times\R^2$,}
    \end{split} 	\tag{H4}
\end{equation}
\begin{equation}\label{H5}
    \begin{split}
    \text{for all $R_1,R_2 \in \R^+, H(x,u,p - \lambda n(x)) \rightarrow +\infty \qquad \lambda \rightarrow +\infty$, } \\
   \text{ uniformly for $(x,u,p) \in\Gamma \times [-R_1,R_2] \times B(0,R_2)$.}
    \end{split} 	\tag{H5}
\end{equation}
Let $u\in \Usc(\overline\O)$ (respectively  $v\in \Lsc(\overline \O)$) be a subsolution (respectively, supersolution) of  \eqref{general_boundary}
with $B(x,r,p)=r-g(x)$. Then $u\le v$ in $\O$.
\end{theo}
Even if the comparison between subsolution and supersolution holds only in $\O$ and not in all $\overline\O$, the previous result implies the uniqueness
of the viscosity solution in $\overline\O$, i.e. if $u_1, u_2\in C(\overline \O)$ are two viscosity solutions of \eqref{general_boundary} then $u_1=u_2$ in $\overline\O$.
Moreover,  Theorem \ref{theorem_strong_uniqueness} is the type of result needed  to prove the convergence of approximation schemes via the Barles-Souganidis' method.\\
Concerning the other boundary conditions which are typically considered in the SfS framework, for the validity of the corresponding Maximum Principle    it is sufficient to assume
\begin{itemize}
  \item[(i)]  for the \textbf{State Constraints  condition} $B(x,r,p)=r-g(x)$ with $g(x)\equiv-\infty$, that $H$ is continuous and satisfies
  (H1)-(H4). Note that in this case the boundary condition reduces to
 \begin{equation}\label{state_constraints}
 H(x,u,\nabla u)\ge 0 \qquad\text{on $\partial \O$}
 \end{equation}
  in viscosity sense and, therefore, it requires no additional information on the solution at the boundary.
\item[(ii)] for the \textbf{Neumann   condition}  $B(x,r,p)=N(x)\cdot p-g(x)$, that the function $g$ is continuous on $\partial \O$ and $H$ is continuous and satisfies
  (H1)-(H3). In this case the Maximum Principle holds in $\overline \O$, i.e.  $u\le v$ in $\overline \O$ in Theorem \ref{theorem_strong_uniqueness}.
\end{itemize}

\section{Well-posedness of the PSfS problem}\label{sec:well-posedness}
Aim of this section is to show that the SfS Hamiltonians $H^L, H^{ON}, H^{PH}, H^{BP}$, introduced in Section \ref{sec:models},
completed with appropriate (either strong or weak) boundary conditions, give raise to well-posed problems (either \eqref{Dirichlet} or \eqref{general_boundary}), 
i.e. problems for which the uniqueness of the viscosity solution holds.
In this section we assume that
\begin{equation}\label{H_lipsch}
\begin{split}
&\text{$\cI(x)$ is Lipschitz continuous in $\overline \Omega$ and there exists}\\
&\text{a neighborhood $\Gamma$ of $\partial \Omega$   such that  $\cI(x)\ge \delta>0$ in $\Gamma$},
\end{split}
\end{equation}
where
\begin{equation}\label{H_reflectance}
 \cI(x)=\left\{
\begin{array}{ll}
 I(x)           \qquad&\text{if $H=H^L, H^{ON}$,}\\[4pt]
 I(x)- k_A I_A(x) &\text{if $H=H^{PH}, H^{BP}$.}
  \end{array}
  \right.
\end{equation}
\subsection{Validation of assumptions for strong  boundary conditions}
We start now to prove that the assumptions \eqref{H1}-\eqref{H2} previously defined in Section \ref{subsec:strong_BC} hold for our Hamiltonians.
\begin{pro}\label{prop}
All the SfS Hamiltonians $H^L, H^{ON}, H^{PH}, H^{BP}$ defined in \eqref{eq:L_Hamiltonian}, \eqref{eq:ON_Hamiltonian}, \eqref{eq:PH_Hamiltonian}, \eqref{eq:BP_Hamiltonian} satisfy the assumptions (H1)-(H2).
\end{pro}
\begin{proof}\\
\noindent {\bf Validation of hypothesis (H1)}\\
Recalling that all the Hamiltonians   $H^{L}$, $H^{ON}$, $H^{PH}$ and $H^{BP}$
can be written in the form $H(x,u,p)= -e^{-2u} + \cH^M(x,p)$ for an appropriate function $\cH^M(x,p)$, see \eqref{HJ1}, then for $u,v$
such that $-R \le v \le u \le R$, we have
\begin{eqnarray*}
H(x,u,p) - H(x,v,p) &=&  -e^{-2u} + \cH^M(x,p) + e^{-2v} - \cH^M(x,p) \\
&=&  -e^{-2u} + e^{-2v}\underset{\text{for some $c \in [v,u]$}}{=}2(u-v) e^{-2c} \\
&\geq&   2(u-v) e^{-2R}= \gamma_R(u-v).
\end{eqnarray*}
 Hence, all the Hamiltonians satisfy assumption (H1) with $\gamma_R = 2 e^{-2R} > 0$.\\
\noindent {\bf Validation of hypothesis (H2)}\\
For $H(x,u,p)$ equal to either $H^{L}$,$H^{ON}$, $H^{PH}$ or $H^{BP}$, by \eqref{HJ1} we can write
\begin{equation}
|H(x,u,p)-H(y,u,p)|= |\cH^M(x ,p)-\cH^M(y ,p)|,
\end{equation}
where $M$ is the acronym of the models ($M = L,ON,PH,BP$).
Recalling \eqref{H_lipsch},
with $\cI$ defined as in \eqref{H_reflectance}, we can write
\begin{align*}
|\cH^M(x,p)-\cH^M(y,p)|=\cI(x)f^2 F_M(W(x,p))-\cI(y)f^2F_M(W(y,p))\le \\|\cI(x)-\cI(y)|f^2F_M(W(x,p))+f^2
|\cI(y)||F_M(W(x,p))-F_M(W(y,p))|\le\\ f^2\|D\I\|_\infty|x-y|F_M(W(x,p))+f^2\|\cI\|_\infty|F_M(W(x,p))-F_M(W(y,p))| .
\end{align*}
Since $ D_x F_M(W(x,p)) =F'_M(W(x,p))D_x W(x,p)$,
we use
\eqref{hp:W_bounded}, \eqref{hp:Dx_W_bounded} and \eqref{Bound_der_F} for $\cH^{L}$,
\eqref{hp:Dx_W_bounded},  \eqref{Bound_der_F_ON} for $\cH^{ON}$,
\eqref{hp:Dx_W_bounded}, \eqref{Bound_der_F} for $\cH^{PH}$ and $\cH^{BP}$ 
to  obtain \eqref{H2} with $C$ depending on $\|\cI\|_\infty$ and $\|D\cI\|_\infty$.
\end{proof}\qed \\
As an immediate corollary of the previous Proposition \ref{prop}  and   Theorem \ref{theorem_weak_uniqueness}, we get
\begin{cor}
The Dirichlet  problem \eqref{Dirichlet} for $H$ given by  $H^L, H^{ON}, H^{PH}  or \, H^{BP}$
admits at most one viscosity solution.
\end{cor}
\subsection{Validation of assumptions for weak  boundary conditions}
Let us see now the proof of the assumptions (H3)-(H5) introduced in Section \ref{subsec:weak_BC}.
\begin{pro} \label{prop2}
We have
\begin{itemize}
  \item[(i)] The SfS Hamiltonians $H^L,H^{PH}, H^{BP}$ defined in \eqref{eq:L_Hamiltonian},  \eqref{eq:PH_Hamiltonian}, \eqref{eq:BP_Hamiltonian}, satisfy the assumptions (H3)-(H5).
  \item[(ii)]  The SfS Hamiltonian $H^{ON}$ defined  in  \eqref{eq:ON_Hamiltonian} satisfies (H3).
\end{itemize}
\end{pro}
\begin{proof}\hfill\\
\noindent {\bf Validation of hypothesis (H3)}\\
For $H(x,u,p)$ equal to either  $H^{L}$, $H^{ON}$, $H^{PH}$ or $H^{BP}$, we have
\begin{align*}
|D_pH(x,u,p) |= |D_p\cH^M(x, p)|=I(x)f^2|D_pF_M(W(x,p))|\\= f^2\|\cI\|_\infty |F'_M(W(x,p))D_p W(x,p)| \, .
\end{align*}
Hence, by \eqref{hp:Dp_W_bounded} and \eqref{H_lipsch}, we get \eqref{H3} with $m_R(t)=Ct$ for any $R\in (0,+\infty)$. \\
%
\noindent {\bf Validation of hypothesis (H4)}\\
For $H(x,u,p)$ equal to $H^{L}$ and $F$
equal to $F_{L}$, we have
\[
H(x,u,p+\lambda n(x))\le 0\Rightarrow F(W(x,p+\lambda n(x)))\le \frac{e^{2R}}{\delta f^2}
\]
for any $u\in (-R,R)$,   $x\in \Gamma$ and $p\in \R^2$. Recalling \eqref{hp:W_bounded} and since $F$ is invertible from
 $[0,+\infty)$ to $[0,+\infty)$, we get
\[|p+\lambda n(x)|\le \sqrt{\frac{1}{f^2}F^{-1}\left(\frac{e^{2R}}{\delta f^2}\right)}=C_R \,. \]
Therefore, we get \eqref{H4} for the Lambertian Hamiltonian $H^L$. The previous argument works exactly in the same way
for  $H(x,u,p)$ equal to either $H^{PH}$  or $H^{BP}$  since the corresponding functions $F_{PH}$ and  $F_{BP}$ satisfy the same properties
 of $F_L$ (see Lemma \ref{lemma2}(i)).
 This is not true for the Hamiltonian $H^{ON}$ defined in  \eqref{eq:ON_Hamiltonian} since $F_{ON}$ is bounded (see Lemma \ref{lemma2}(ii)). \\
\noindent {\bf Validation of hypothesis (H5)}\\
For $H(x,u,p)$ equal to $H^{L}$ and $F$ equal to $F_{L}$, we have
\begin{eqnarray}
H(x,u,p-\lambda n(x)) &\ge& e^{-2R} + \delta f^2 F(W(x,p-\lambda n(x))) \nonumber \\
&=& e^{-2R} + \delta f^2 \sqrt{W(x,p-\lambda n(x)) +1} \nonumber \\
&\ge& e^{-2R} + \delta f^2 \sqrt{f^2 |p-\lambda n(x)|^2} \nonumber \\
&=& e^{-2R} + \delta f^2 f |p-\lambda n(x)|
\end{eqnarray}
that goes to $+\infty$ for $\lambda \rightarrow +\infty$. Hence, we get \eqref{H5} for  the Lambertian Hamiltonian \eqref{eq:L_Hamiltonian}.\\
The previous argument applies also   to the Phong Hamiltonian  \eqref{eq:PH_Hamiltonian}  since for $W(x,p) \ge 1$ the behavior of $H^{PH}$ is as the one of $H^L$. \\
For the BP--model, considering the associated Hamiltonian $H^{BP}$ defined in \eqref{eq:BP_Hamiltonian} we have
\begin{eqnarray}
H^{BP}(x,u,p-\lambda n(x)) &\ge& e^{-2R} + \delta f^2 F^{BP}(W(x,p-\lambda n(x))) \nonumber \\
&=& e^{-2R} + \delta f^2 \frac{(W(x,p-\lambda n(x)) +1)^{\frac{c}{2}}}{k_D (W(x,p-\lambda n(x)) +1)^{\frac{c-1}{2}} +k_S} \nonumber \\
&\ge& e^{-2R} + \delta f^2 \frac{(W(x,p-\lambda n(x)) +1)^{\frac{c}{2}}}{(k_D+k_S) (W(x,p-\lambda n(x)) +1)^{\frac{c-1}{2}}} \nonumber \\
&=& e^{-2R} + \delta f^2 \frac{\sqrt{W(x,p-\lambda n(x)) +1}}{(k_D+k_S)} \nonumber \\
&\ge& e^{-2R} + \delta f^2 \frac{\sqrt{f^2 |p-\lambda n(x)|^2}}{(k_D+k_S)} \nonumber \\
&=& e^{-2R} + \delta f^2 \frac{f |p-\lambda n(x)|}{(k_D+k_S)}
\end{eqnarray}
that goes to $+\infty$ for $\lambda \rightarrow +\infty$. Hence, we get \eqref{H5} for $H^{BP}$.
\end{proof}\qed \\
As a consequence of the previous Proposition \ref{prop2}  and   Theorem \ref{theorem_strong_uniqueness}, we get the following
\begin{cor} \label{SfS_strong_uniqueness}
We have
\begin{itemize}
 \item[(i)] For $H$ given by  $H^L,  H^{PH}$  or $H^{BP}$, the boundary value problem \eqref{general_boundary}
with   either Dirichlet, Neumann or state constraints boundary conditions in viscosity sense satisfies
 \emph{the Maximum Principle for discontinuous viscosity solutions} and therefore it admits at most one viscosity solution.
\item[(ii)] For $H$ given by  $H^{ON}$, the boundary value problem \eqref{general_boundary}
with     Neumann boundary condition in viscosity sense satisfies
  \emph{the Maximum Principle for discontinuous viscosity solutions} and therefore it admits at most one viscosity solution.
\end{itemize}
\end{cor}
We do not discuss here   the existence of a viscosity solution to the boundary value problems \eqref{Dirichlet}
and \eqref{general_boundary}. Since the SfS problem is an inverse problem in Computer Vision, we a priori know that, if  we  prescribe the correct  boundary conditions, the solution $u$ is given by the height of the surface to be reconstructed.   Moreover, thanks to the Maximum Principle for discontinuous viscosity solutions, see Theorem \ref{theorem_strong_uniqueness}, it is possible to get existence of a solution by proving the convergence of the various schemes
considered in literature (see e.g. \cite{BCDFV12,DFS08}). Finally, the Maximum Principle is also a key ingredient to get existence of a solution via Perron's method, once a subsolution  and a supersolution of the problem have been constructed (see \cite{CIL}).\par
We conclude with some remarks about the different boundary conditions in Shape-from-Shading  theory. The choice of the boundary condition
obviously depends on the chosen model and on available data. But, as observed in \cite{PF05}, the advantage of the state constraints boundary condition with respect to the   Dirichlet and Neumann boundary conditions is that it does not require any a priori   information about the solution at the boundary.
Nevertheless, the  approximation of  this boundary condition can be difficult. In the next lemma we show that the solution of a Dirichlet problem with a large constant as boundary data and the solution of state constraints  problem are upper  bounded by the same constant $M$ (independent of the solution).
Therefore, in this case, by the very definition of  boundary conditions in weak sense,  the Dirichlet supersolution condition    reduces to the state constraint one \eqref{state_constraints}.
 It follows that to approximate a problem with state constraints boundary condition is sufficient to consider the corresponding problem  with Dirichlet  boundary datum given by a constant  $M$ sufficiently large.
 This request is not a strong constraint difficult to occur since, for example, if the input image contains an object of interest in front of a background, the condition is satisfied in a neighbourhood of the object where $u(x)$ increases rapidly. 
\begin{lemma}\label{lemma3}
For $H(x,u,p)$ equal to either  $H^{L}$,  $H^{PH}$, or $H^{BP}$,
if the intensity image $\cI$  in \eqref{H_reflectance}   verifies
\[ \cI(x)\ge \delta > 0,\quad \forall x \in \Omega,\]
then the solution of the Dirichlet problem \eqref{Dirichlet} with an appropriately  large constant as boundary data
and the solution of the problem \eqref{general_boundary} with state constrains conditions  are upper bounded by the same constant (dependent only on the data of the problem).
\end{lemma}
\begin{proof}
For $H(x,u,p)$ equal to $H^{L}$, let us define $M^L = - \frac{1}{2}\ln (\delta f^2)$. Then, we have
\[
H^L(x,M^L,0) = -e^{-2M^L} + I(x) f^2 \ge 0 \qquad \forall x\in\Omega.
\]
Therefore, the constant function $u(x)\equiv M^L$ is a supersolution of \eqref{Dirichlet} in $\Omega$.
Moreover, $\forall x \in \partial\Omega$ and $\forall \xi \in \R^2$, we have
\begin{equation}
H^L(x,M^L,\xi) = -e^{-2M^L} + I(x) f^2 \sqrt{W(x,\xi) + 1} \ge -e^{-2M^L} + I(x) f^2 \ge 0,
\end{equation}
which implies that $u(x)\equiv M^L$ is a supersolution of \eqref{HJ} on the boundary and therefore, see \eqref{state_constraints}, a supersolution of the state constraint boundary problem in $\overline \O$. Hence, by Theorem \ref{theorem_strong_uniqueness}  and Corollary \ref{SfS_strong_uniqueness}, the
solutions of the Dirichlet and state constraints problems with Hamiltonian $H^L$ are upper bounded by the same constant $M^L$.

We now prove the same result for the Phong model. Let us define $M^{PH} = - \frac{1}{2}\ln (\frac{\delta f^2}{k_D + k_S})$.\\
We have
\begin{eqnarray}
H^{PH}(x,M^{PH},0) &=& -e^{-2M^{PH}} + \frac{(I(x) - k_A I_A(x) ) f^2}{k_D + k_S} 	\\
&=& -\frac{\delta f^2}{k_D + k_S} + \frac{(I(x) - k_A I_A(x) ) f^2}{k_D + k_S}
\end{eqnarray}
that is non-negative since $\cI(x) = (I(x) - k_A I_A(x) ) \ge \delta > 0, \,\, \forall x \in \Omega$. \\
Therefore, the constant function $M^{PH}$ is a supersolution of \eqref{general_boundary} in $\Omega$ for the PH--model.
Moreover, $\forall x \in \partial\Omega$ and $\forall \xi \in \R^2$, we have \\
for $W(x,p) > 1$,
\begin{eqnarray*}
H^{PH}(x,M^{PH},\xi) &= &   -e^{-2M^{PH}} + (I(x) - k_A I_A(x) ) f^2 \sqrt{W(x,\xi) + 1}\\
                            &\ge& -e^{-2M^{PH}} + (I(x) - k_A I_A(x) ) f^2 \ge 0;
\end{eqnarray*}
for  $0 \le W(x,p) \le 1$,
\begin{eqnarray*}
H^{PH}(x,M^{PH},\xi) &=& -e^{-2M^{PH}}\\
   & +& (I(x) - k_A I_A(x) ) f^2 \frac{(W(x,\xi)+1)^{\alpha + 1/2}}{k_D (W(x,\xi)+1)^{\alpha} + k_S (W(x,\xi)+1)^{1/2} (1-W(x,\xi))^{\alpha}}  \\
   &\ge&  -e^{-2M^{PH}} + \frac{(I(x) - k_A I_A(x) ) f^2}{k_D + k_S}  \ge 0.
\end{eqnarray*}
So,  the constant function $M^{PH}$ for the PH--model is a supersolution of \eqref{general_boundary} in $\overline\Omega$ for both the
Dirichlet problem with boundary datum $M^{PH}$ and the state constraints one.

For the BP--model, we can consider the same constant used for the PH--model and repeat the same steps to prove the Lemma.
\end{proof}\qed


\section{Conclusions}\label{sec:conclusions}
In this paper we have shown the well-posedness of several non-Lambertian models in an analytical way thanks to the introduction of an attenuation term, under a perspective projection with a unique light source located at the optical center of the camera. 
We have derived the Hamilton-Jacobi equations associated to each reflectance model and then we have obtained a  unified   formulation showing useful properties for the general Hamiltonian. We have clarified the differences between the various models from the point of view of the problem with Dirichlet boundary conditions (BC), Neumann BC and state constraints BC. This analysis explains (and corrects) in details the empirical deductions coming from previous works on non-Lambertian models (see e.g. the conclusion reported in \cite{AF06} regarding the concave/convex ambiguity for the ON-model), focusing the attention on the important role of the boundary conditions
and providing also an useful theoretical tool for the numerical point of view thanks to the last Lemma \ref{lemma3}.\\
It is also worth noting that, from a theoretical point of view, contrary to previous works (see e.g. \cite{DO94,OD93,HB89,CDD03}),  we do not require  any additional information for the models and we do not  need to add regularization terms to the PDEs for ensuring the well-posedness of the SfS problem. 
From a numerical point of view,  the unified approach allows  to consider a unique and general numerical scheme that incorporates all the different SfS   models described in the paper  (and possibly  other non-Lambertian models not considered here), thereby providing  an easy way to   compare the performances of the different models by only changing some parameters.

\appendix
\section{Proofs of lemmas in Subsection \ref{sec:unified_formulation} } \label{App:AppendixA}
\begin{Proofc}{Proof of Lemma \ref{lemma1}.}
We have
\[f^2|p|^2\le W(x,p)=|p|^2 W(x,\frac{p}{|p|})\le |p|^2\max_{|q|=1} W(x,q)\]
which proves \eqref{hp:W_bounded}. Moreover,
\begin{align*}
    D_x W(x,p)=2(p\cdot x)p\frac{|x|^2+f^2}{f^2}+\Big[(p\cdot x)^2 + f^2 |p|^2\Big] \frac{2x}{f^2} \\
    D_p W(x,p)=2(f^2 + |x|^2) p+2\frac{|x|^2+f^2}{f^2}(x\cdot p)x
\end{align*}
from which we get  \eqref{hp:Dx_W_bounded} and \eqref{hp:Dp_W_bounded}.
\end{Proofc}
\begin{Proofc}{Proof of Lemma \ref{lemma2}.}\\
(i)
For the function $F_L$ we have
\begin{eqnarray*}
   \lim_{r\to+\infty} F_L(r) = \lim_{r\to+\infty}  \sqrt{r+1} =  +\infty,  
    F'_{L} =\frac{1}{2\sqrt{r+1}} \quad  \textrm{for $r\in [0,+\infty)$,} 
\end{eqnarray*}
and the corresponding properties in the statement easily follow.\\
Regarding the PH--model, we analyze the function $ F_{PH}(r)$ only for    $r\in [0,1)$ (otherwise it is equal to $\displaystyle \frac{F_L(r)}{k_D}$).
For $r\in [0,1)$, we have
\begin{equation*} 
\begin{split}
   F'_{PH}(r) =\frac{1}{[k_D(r+1)^\alpha+k_S(r+1)^{1/2}(1-r)^\alpha]^2}\cdot\\
   \Big[(\frac{1}{2}+\alpha)\Big(k_D(r+1)^{2\alpha-1/2}+k_S(r+1)^{\alpha} (1-r)^{\alpha}\Big)-\alpha k_D(r+1)^{2\alpha-1/2}\\
 -k_S\Big(\frac{1}{2}(r+1)^\alpha(1-r)^\alpha-\alpha(r+1)^{\alpha+1}(1-r)^{\alpha-1}   \Big)\Big]  \\
 = \frac{1}{[k_D(r+1)^\alpha+k_S(r+1)^{1/2}(1-r)^\alpha]^2}\cdot
   \Big[\frac{1}{2} k_D(r+1)^{2\alpha-1/2} + 2 \alpha k_S(r+1)^{\alpha} (1-r)^{\alpha-1}\Big].
   \end{split}
\end{equation*}
Hence,
\begin{eqnarray*}
0 < F'_{PH}(r) &\le&  \frac{1}{k^2_D(r+1)^{2 \alpha}} \Big[\frac{1}{2} k_D(r+1)^{2\alpha-1/2} + 2 \alpha k_S(r+1)^{\alpha}\Big] \\
& = & \frac{1}{2 k_D \sqrt{r+1}} + \frac{2\alpha k_S}{k^2_D (r + 1)^{\alpha}}\\
 & \le &  \Big( \frac{1}{2 k_D} + \frac{2\alpha k_S}{k^2_D} \Big) \frac{1}{\sqrt{r + 1}}.
\end{eqnarray*}
Moreover,
\begin{eqnarray*}
   \lim_{r\to1^-} F'_{PH}(r) =   \lim_{r\to1^+} F'_{PH}(r) = \frac{1}{k_D }F'_L(1),		
\end{eqnarray*}
hence $F_{PH} \in C^1$.
By the corresponding properties for $F_L$,  it follows that $F_{PH}$ is strictly increasing and the following bounds hold
\begin{align*}
   \text{$0<F'_{PH}(r) \le C\frac{1}{\sqrt{r+1}} \quad$ for $r\in[0,+\infty)$ and $\forall \alpha \ge 1$.}   \label{Bound_der_F_PH}
\end{align*}
For the BP--model,
if   $c=1$ in \eqref{F_BP}, then  $F_{BP}(r) = \frac{1}{k_D + k_S} F_L(r)$, hence we can refer to $F_L$.
In the case $c > 1$, the function $F_{BP}$ is still smooth and positive. Moreover,
\begin{equation*}
\lim_{r\to \infty} F_{BP} (r)=+\infty
\end{equation*}
and
\begin{eqnarray*}
    F'_{BP}(r) &=& \frac{\frac{c}{2} (r+1)^{\frac{c}{2} -1} (k_D(r+1)^{\frac{c-1}{2}} +k_S) -(r+1)^{\frac{c}{2}} (\frac{c-1}{2} k_D (r+1)^{\frac{c-1}{2} -1})}{(k_D (r+1)^{\frac{c-1}{2}} +k_S)^2}  \\
    &\le& \frac{\frac{c}{2} (r+1)^{\frac{c}{2} -1} (k_D(r+1)^{\frac{c-1}{2}} +k_S) }{(k_D (r+1)^{\frac{c-1}{2}} +k_S)^2} =
      \frac{\frac{c}{2} (r+1)^{\frac{c}{2} -1}}{k_D (r+1)^{\frac{c-1}{2}} +k_S} \, .
\end{eqnarray*}
Hence,   the following bounds hold
\begin{align*}
   \text{$0<F'_{BP} (r) \le C\frac{1}{\sqrt{r+1}} \quad$ for $r\in[0,+\infty)$ and $\forall c \ge 1$.}   \label{Bound_der_F_BP}
\end{align*}
This shows that also the function $F_{BP}$ is strictly increasing, that concludes the proof of (i).\\
(ii)
For the function $F_{ON}$ we have
\begin{equation*}
   F'_{ON}(r)=\frac{A\sqrt{r+1}-2B}{2[A\sqrt{r+1}+Br]^2}\quad \textrm{for $r\in [0,+\infty)$.}		\label{der_F_ON}
 \end{equation*}
Hence, $F_{ON}$ is strictly increasing if $A/2 > B$ and the property \eqref{Bound_der_F_ON} in the statement easily follows. Moreover, since $\lim_{r\to+\infty} F_{ON}(r) =1/B$, it is bounded.
\end{Proofc}

\bibliographystyle{plain}
\bibliography{references_2}

\end{document}